\documentclass[11pt]{article}
\usepackage{amsmath}
\usepackage{mathtools}
\usepackage{amssymb}
\usepackage{mathabx}
\usepackage{graphicx}
\usepackage{amsthm}
\usepackage{multicol}
\usepackage{geometry}
\usepackage{graphicx}
\usepackage{hyperref}
\usepackage{setspace}
\usepackage{enumerate}
\usepackage{xcolor}
\usepackage{tikz}
\usepackage{algorithm2e}
\usepackage{pgfplots}
\usepackage{tikz-3dplot}
\usepackage{longtable}

\newtheorem{lem}{Lemma}[section]
\newtheorem{thm}[lem]{Theorem}
\newtheorem{prop}[lem]{Proposition}
\newtheorem{cor}[lem]{Corollary}
\newtheorem{defn}[lem]{Definition}
\newtheorem{exm}[lem]{Example}
\newtheorem{rmk}[lem]{Remark}

\newtheorem{prob}[lem]{Problem}

\newcommand{\kw}{\noindent\textbf{Keywords: }}
\newcommand{\msc}{\noindent\textbf{2020 MSC: }}
\newcommand{\p}{\noindent \mathrm{P}}
\newcommand{\pr}{\noindent \mathrm{Prox}}

\newcommand{\N}{\noindent \mathrm{N}}
\newcommand{\T}{\noindent \mathrm{T}}

\newcommand{\RR}{\noindent \mathbb{R}}
\newcommand{\NN}{\noindent \mathbb{N}}

\newcommand{\HH}{\noindent \mathcal{H}}
\newcommand{\II}{\noindent \mathbb{I}}
\newcommand{\ii}{\noindent \iota}

\newcommand{\dom}{\noindent \mathrm{dom}}
\usepackage{tikz}
\usepackage{eso-pic}

\begin{document}
	\centerline{}
	\centerline{\Large{\bf 
			Exact and Approximate Solvability of Systems Involving}}
	\centerline{\Large{\bf 
		 Proximity Operators in Hilbert Spaces}}
	{\singlespacing
		\centerline{\bf {Mark Allien D. Roble}}
		\centerline{Institute of Mathematical Sciences, University of the Philippines Los Baños, Philippines}
		\centerline{{\it mdroble@.up.edu.ph}}
	}
	\centerline{} 
	\begin{abstract}
		\baselineskip=1\baselineskip
		Let $\HH$ be a real Hilbert space and let $(f_i)_{i\in I}$ be a finite family of proper, lower semicontinuous, and convex functions on $\HH$. This study investigates the existence and uniqueness of exact solutions to systems involving proximity operators of the form: \
		$(\forall i\in I)\ \pr{f_i}(x)=p_i,$
		where $(p_i)_{i\in I}$ is a prescribed collection of proximal points in $\HH$, which naturally generalize classical projection problems. We establish necessary and sufficient conditions for the existence of approximate solutions to such systems. Moreover, we introduce and derive several characterizations of the inverse proximal property (IPP), as a generalization of the inverse best approximation property (IBAP). Applications of the obtained results are presented in the contexts of a feasibility problem and signal recovery problem, demonstrating the relevance of the proposed framework to optimization and inverse problems in Hilbert spaces.
	\end{abstract}
	
	\kw{convex functions, subdifferential, proximity operator, inverse proximal property, inverse best approximation, feasibility problem}\newline
	\msc{26B25, 41A35, 47N10, 52A27}
	
	\section{Introduction}

The problem of finding a point that satisfies multiple constraints simultaneously is a central theme in optimization and variational analysis. In its classical form, this task appears as the convex feasibility problem, which seeks a point belonging to the intersection of a family of closed convex sets. Such problems arise naturally in signal processing, image reconstruction, image denoising, tomography, and inverse problems, where each constraint represents a piece of available information about the desired solution \cite{carmi, cens, heat, herm, convexfeas}. Over the years, this problem has emerged as a powerful and versatile mathematical framework, finding extensive applications in diverse fields such as medical physics, engineering, communication, and economics \cite{abol, bark, garc, ghol}. Owing to its broad applicability, numerous researchers have developed solution methods based on projection operators.

For a nonempty, closed, and convex subset $C$ of $\mathcal{H}$, the {\it projection operator onto $C$} at $x\in\HH$ is defined as $\p_Cx =\arg\min_{y\in C} \hspace{.1cm} \Vert x-y\Vert$. This projection can be traced back to the work of Zarantonello, who found its relevant application in the spectral theory \cite{zar}. Motivated by the fundamental role of projections in feasibility problems, Combettes and Reyes investigated in 2010 the system of orthogonal projections \cite[Eqn.(1.1)]{co}. They coined the term, \emph{inverse best approximation property} (IBAP) for a finite collection of nonempty closed linear subspaces $(C_i)_{i\in\II}$, if for every $(p_i)\in\bigtimes_{i\in\II} C_i$, there exists a vector $x\in\mathcal{H}$ satisfying the system
\begin{equation}
	\label{systemconvex}
	(\forall i\in\II)\ \p_{C_i}x =p_i.
\end{equation}
In their seminal work, they established several equivalent characterizations of the IBAP and identified necessary and sufficient conditions under which a family of closed linear subspaces satisfies this property \cite[Theorem 2.8]{co}.

A significant generalization of the convex feasibility problem is provided by the theory of monotone operators (see \cite[Chapters 20-25]{ba} for more details). These operators play a fundamental role in modern optimization because many mathematical models, including convex minimization problems, variational inequalities, equilibrium problems, and partial differential equations, can be formulated as monotone inclusions \cite{bene, guan, zhou}. Let $\mathcal{H}$ be a finite (or infinite) dimensional real Hilbert space, and let $\II$ be a nonempty finite index set. The compatibility problem consists of determining whether there exists a point $x\in\mathcal{H}$ that is simultaneously compatible with all the operators, in the sense that 
\begin{center}
	$\displaystyle 0\in \bigcap_{i\in\II} A_i(x)$, where $\{A_i\}_{i\in\II}$ is a collection of maximally monotone operators on $\HH$.
\end{center}
The latter formulation is commonly referred to as the \emph{monotone inclusion problem} and provides a unifying framework for a wide range of problems arising in convex optimization, variational inequalities, equilibrium theory, and feasibility problems. The compatibility problem has recently attracted attention as a unifying framework for studying systems of monotone operators and their associated fixed-point structures.

In this study, we focus on a special class of maximally monotone operators arising naturally in convex analysis. To this end, we recall several fundamental concepts taken from various convex analysis books \cite{ba, boyd, app8, ro2}. For a given function $f:\HH\longrightarrow(-\infty,+\infty]$, the {\it domain} of $f$ is the set $\dom f=\{x\in\HH \ | \ f(x)<+\infty\}$. We say that $f$ is {\emph{proper}} whenever $\dom f\neq\varnothing$; while it is called a {\emph{convex function}} if $\dom f$ is a convex set, and $$(\forall\alpha\in [0,1])\,(\forall x,y\in\dom f)\ f(\alpha x+(1-\alpha)y)\leq\alpha f(x)+(1-\alpha) f(y).$$
Moreover, $f$ is called a {\emph{lower semicontinuous}} if for each $\eta\in\RR$, the sublevel set $\{x\in\HH \ | \ f(x)\leq \eta\}$ is closed in $\HH$. We denote by $\Gamma_0(\HH)$ the set of all proper, convex, and lower semicontinuous functions defined on $\HH$. A classical example of a function in $\Gamma_0(\HH)$ is the \emph{indicator function} of a nonempty closed convex set $C$, defined to be $\ii_C(x)=0$ if $x\in C$, and $\ii_C(x)=+\infty$ if $x\notin C$. This example plays a central role in convex analysis and serves as a key building block in the study of monotone operators.

Let $\langle \cdot,\cdot \rangle$ and $\Vert\cdot\Vert$ denote the inner product and the norm induced by it on $\HH$, respectively. We now formally introduce the proximity operator, which serves as the central object of this study.
\begin{defn}{\rm \cite[Definition 12.23]{ba}} \rm
	Let $f\in\Gamma_0(\mathcal{H})$. The {\it proximity operator}, also known as {\it proximal mapping}, $\pr_f:\mathcal{H}\longrightarrow\mathcal{H}$, is defined by
	\begin{eqnarray}
		\label{cha:eq2}
		\pr_fx=\displaystyle\arg\min_{y\in\mathcal{H}} \hspace{.1cm}\left\lbrace f(y)+\mathfrak q(y-x) \right\rbrace ,
	\end{eqnarray}
\end{defn}
\noindent where $\mathfrak q(\cdot)=\frac{1}{2}\Vert\cdot\Vert^2$. This mapping was first introduced by Moreau \cite{ro1}, and has been widely used in optimization algorithms involving signal recovery and signal processing \cite{app1, app11}, deep neural network \cite{compes2},  stochastic control and large-data sets \cite{app6}, matrix decomposition \cite{app2}, image recovery \cite{app3}, equilibrium problems \cite{app4, app7}, and depth map estimation \cite{app10}. Interestingly,
\begin{equation*}
	\pr_{\iota_C}x= \arg\min_{y\in \mathcal H} \left\{\iota_C(y) + \mathfrak q(y-x)\right\} = \arg\min_{y\in C} \mathfrak q(y-x)= \arg\min_{y\in C} \Vert y-x\Vert=\p_Cx
\end{equation*}
which consequently shows that proximity operator is a natural extension of projection operator. Here, we call $\pr_fx$ as the {\it  proximal point} of $x$ relative to $f$, which is a generalization of the notion of a best approximation proposed by Moreau \cite{more}. This motivates us to extend system \eqref{systemconvex} to a framework involving proximity operators. In addition, one of the main objectives of this study is to extend the inverse best approximation property (IBAP) for linear subspaces to the framework of proximity operators. We refer to this generalization as the {\it inverse proximal property} (IPP), which will be formally introduced in Section \ref{sec3}.

Before we formulate our main problem, we introduce a set-valued operator called \emph{subdifferential} of $f$, defined as
$$\partial f(x)=\left\lbrace u\in\mathcal{H} \ | \ \forall y\in\mathcal{H}, \langle y-x, u\rangle +f(x)\leq f(y)\right\rbrace.$$
It is a well-known fact that $\partial f$ is maximally monotone \cite[Theorem 21.2]{ba}. In addition, for $f\in\Gamma_0(\mathcal{H})$, the following characterization provides a remarkable relationship between the proximity operator and the subdifferential \cite[Proposition 16.44]{ba},
\begin{equation}
	\label{proxsubd}
	(\forall x\in\mathcal{H})\ (\forall p\in\mathcal{H})\ \pr_f x=p \Leftrightarrow x-p\in\partial f(p).
\end{equation}
The characterization demonstrated in \eqref{proxsubd} is a generalization of the projection theorem \cite[Theorem 3.16]{ba}. For a fix $p\in\HH$ and for $f\in\Gamma_0(\HH)$, the equivalence statement \eqref{proxsubd} guarantees the existence of $x\in\HH$ in our equation whenever $p\in\dom\,\partial f$. With these preliminaries in place, we are now in a position to state the main problem of this study.\\

{\bf Main Problem:}
For each $i\in\mathbb{I}$, let $f_i\in\Gamma_0(\mathcal{H})$ and $p_i\in\mathrm{dom\ }\partial f_i$. Find a vector $x\in\mathcal{H}$ such that 
\begin{equation}
	\label{main_prob}
	\hspace{2in}	(\forall i\in \mathbb{I}) \ \pr_{f_i}x=p_i
\end{equation}
To motivate our analysis, we introduce the \emph{support function} of a set $C\subseteq\mathcal H$, defined by
$\sigma_C(x)=\sup_{c\in C}\langle c,x\rangle.$
According to \cite[Corollary 13.38 and Example 13.43(i)]{ba}, the support function $\sigma_C$ belongs to $\Gamma_0(\mathcal H)$ precisely when $C$ is nonempty, closed, and convex. The following example provides a simple instance of the Main Problem.
\begin{exm}\rm
	\label{main_example}
	Take $\mathcal{H}=\mathbb{R}^2$ and consider the closed convex set $C=[0,1]\times\{0\}$. Find an $x\in\mathcal{H}$ satisfying the system of nonlinear equations
	\begin{equation}
		\label{examsystem}
		\pr_{\iota_C}x=(1,0)\ \mathrm{and}\ \pr_{\sigma_C} x=(1,1).
	\end{equation}
	This problem can be re-posed as finding $x\in\mathcal{H}$ such that $\p_Cx=(1,0)$ and $\pr_{\sigma_C} x=(1,1)$. Let us show that $x=(2,1)$ is a solution to \eqref{examsystem}. For any $x=(x_1,x_2)\in\mathcal{H}$,
	$$\sigma_C(x)=\begin{cases}
		0, \hspace{.55in} \text{if} \ \ x_1< 0\\
		x_1, \hspace{.47in} \text{if} \ \ x_1\geq 0.
	\end{cases}$$
	By the equivalence \eqref{proxsubd}, $\p_Cx=(1,0)$ if and only if $\langle (x_1,x_2)-(1,0),(c,0)-(1,0)\rangle\leq 0$, for all $c\in [0,1]$. But $\langle(x_1-1,x_2),(c-1,0)\rangle\leq 0$, for all $c\in [0,1]$. Thus, $\p_Cx=(1,0)$ if and only if $x_1-1\geq 0$. With this observation, we must have $\p_C(2,1)=(1,0).$
	Meanwhile,
	\begin{center}
		$\sigma_C(1,1)=1=\langle (1,0), (1,1)\rangle$, and for any $y\in\mathcal{H}$, $\sigma_C(y)\geq\langle (1,0), y\rangle$. 
	\end{center}
	Consequently, for any $y\in\mathcal{H}$,
	\begin{align*}
		\sigma_C(1,1)+\dfrac{1}{2}\Vert (1,0)\Vert^2& \leq \langle(1,0),(1,1)\rangle+\dfrac{1}{2}\Vert (1,0)\Vert^2+\sigma_C(y)-\langle (1,0),y\rangle+\dfrac{1}{2}\Vert (1,1)-y\Vert^2\\
		&\leq\sigma_C(y)+\dfrac{1}{2}\Vert (1,0)+( (1,1)-y)\Vert^2.
	\end{align*}
	This further implies that $\sigma_C(1,1)+\mathfrak q\left( (2,1)-(1,1)\right)\leq\sigma_C(y)+\mathfrak q\left((2,1)-y\right)$. In turn, $\pr_{\sigma_C}(2,1)=(1,1)$, and we have shown that
	\begin{center}
		$\p_C(2,1)=(1,0)$ and $\pr_{\sigma_C}(2,1)=(1,1)$.
	\end{center}
\end{exm}

Due to \eqref{proxsubd}, a particular instance of the compatibility problem arises from a system involving proximity operators. Since $\pr_{f_i}=(I+\partial f_i)^{-1}$ \cite[Eqn.(16.38)]{ba}, the system \eqref{main_prob} is equivalent to finding $x\in\HH$ such that
`\begin{center}
	$x-p_i\in\partial f_i(p_i),\ i\in\II,$ that is, $0\in \partial f_i(p_i)+p_i-x,\ i\in\II.$
\end{center}
Defining the shifted operators, $A_i(x):=\partial f_i(p_i)+p_i-x$, which are maximally monotone operators, the Main Problem reduces to finding a point $x\in\HH$ such that
$$
0\in \bigcap_{i\in\II} A_i(x).
$$
Thus, the existence of a solution to the system \eqref{main_prob} can be interpreted as a compatibility problem for a family of maximally monotone operators. Despite its close connections to monotone operator theory and convex optimization, this problem remains relatively underexplored. The present work aims to contribute to filling this gap. The investigation of systems involving proximity operators is particularly significant because such systems capture the interplay among multiple convex functions. They naturally arise in various settings, including decomposition algorithms, multi-block optimization frameworks, and consensus-based models \cite{kiwi, zeng, hu}. Consequently, understanding the existence, characterization, and approximation of compatible proximal points can provide valuable insights into both the theoretical foundations and computational methodologies of convex optimization. Furthermore, compatibility conditions for proximity operators often uncover underlying geometric and variational structures of the associated convex functions, thereby deepening our understanding of their collective behavior.

In 2021, Combettes and Woodstock designed an algorithm that would find the solution (assuming that it exists) of the variational formulation 
\begin{equation} {\rm minimize } \ \|x-x_0\| \ \text{ subject to } \ x\in \bigcap_{j\in J} C_j \ \text{ and } \ (\forall k\in K) \ \pr_{f_k}x=p_k, 
	\label{comwoodeq}
\end{equation}
where $C_j, j\in J,$ are closed convex subsets of $\mathcal H$ with prescribed proximal points $(p_k)_{k\in K}$ of $x$ relative to $(f_k)_{k\in K}$ in $\Gamma_0(\mathcal H)$ \cite[Problem 1.1]{comwood}. The system of constraints in \eqref{comwoodeq} is a particular instance of \eqref{main_prob}; this connection will be established in Problem \ref{specialwood}. In the same year, the authors examined the system \eqref{main_prob} in a more general framework by replacing proximity operators with firmly nonexpansive operators to address a signal recovery task \cite{comwood2}. They reformulated the problem as a fixed-point equation and solved it using a block-iterative algorithm. The following year, a relaxed version of the original problem was proposed in the form of a variational inequality \cite{comwood3}, where the block-iterative method was again employed, activating only selected blocks of the underlying firmly nonexpansive operators at each iteration. In this study, we investigate the existence and uniqueness of solutions to the system \eqref{main_prob} from an analytical standpoint, focusing on theoretical foundations rather than algorithmic techniques. Finally, in 2023, Roble and Vallejo considered the system \eqref{systemconvex} for a finite collection of nonempty closed convex sets. They provided conditions for the system to admit a solution \cite[Proposition 1.2]{robval} and present a characterization of when a pair of nonempty closed convex sets $(C,D)$ has the IBAP \cite[Section 3]{robval}. We shall see in this study that the IBAP admits a natural extension to the IPP.

We summarize this paper as follows. In Section \ref{sec2}, we provide a characterization for the existence and uniqueness of solutions to the system \eqref{main_prob}. Furthermore, we investigate the existence of approximate solutions to the system \eqref{main_prob}, and present the necessary and sufficient conditions for the existence of approximate solutions. In Section \ref{sec3}, we introduce the IPP, showing that it naturally extends the IBAP to the setting of proximity operators. We also present several consequences and implications of this property. Finally, in Section \ref{sec4}, the results obtained will be applied to feasibility problem and signal recovery problem.
	
	\section{Consistency of the System of Proximity Operators}
\label{sec2}

Assume $(f_i)_{i\in\II}$ to be a finite collection of functions in $\Gamma_0(\HH)$, and $(p_i)_{i\in\II}$ are the proximal points of $x$ relative to $(f_i)_{i\in\II}$. Given a set $A$ and a point $p\in\HH$, the {\it translation} of $A$ by $p$ is defined as $p+A:=\{p+a\ \vert\ a\in A\}$. As an immediate consequence of \eqref{proxsubd}, the first proposition characterizes the existence of solutions to the system \eqref{main_prob} in terms of the consistency of the translations of subdifferentials induced by proximal points.
\begin{prop}
	\label{exist}
	Let $(p_i)_{i\in\II}\in\bigtimes_{i\in\II}\dom\,\partial f_i$. For any $x\in\HH$,
	\begin{center}
		$(\forall i\in \mathbb{I}) \ \pr_{f_i}x=p_i\Leftrightarrow x\in \bigcap_{i\in \mathbb{I}} \left( p_i + \partial f_i(p_i) \right).$
	\end{center}
\end{prop}
In line with \eqref{proxsubd}, the uniqueness of solution to the equation $\pr_fx=p$ is guaranteed whenever $\partial f(p)$ is a singleton set. We shall see in the next result that G\^ateaux differentiability is a sufficient condition for the uniqueness of solution to the system \eqref{main_prob}. This can be hinged on \cite[Proposition 17.31(i)]{co}.
\begin{prop}
	\label{Gatuni}
	Let $(p_i)_{i\in\II}\in\bigtimes_{i\in \mathbb{I}}\dom\,\partial f_i$. Suppose that for some $k\in \mathbb{I}$, $f_k$ is G\^ateaux differentiable at $p_k$. Then the system \eqref{main_prob} admits at most one solution in $\mathcal{H}$. Moreover, if there exists a solution $x\in \mathcal{H}$ to \eqref{main_prob}, then $x=p_k+ \nabla f_k(p_k)$.
\end{prop}
\begin{proof}
	In view of Proposition \ref{exist}, the assumption that $f_k$ is G\^ateaux differentiable at $p_k$, the set $$S=\bigcap_{i\in \mathbb{I}\setminus \{k\}} \left( p_i + \partial f_i (p_i)\right) \cap \{p_k+ \nabla f_k(p_k)\}$$ of solutions to \eqref{main_prob} contains at most one element. In case $S\neq \varnothing$, then $S$ contains only the element $p_k+ \nabla f_k(p_k)$.   
\end{proof}

As a corollary to the Proposition \ref{Gatuni}, we provide a characterization for the uniqueness of solutions to the system involving two proximity operators.
\begin{cor}
	\label{exampleuni}
	Let $f,g\in\Gamma_0(\mathcal{H})$ and $(p,q)\in\dom\,\partial f\times\dom\,\partial g$. Suppose $f$ and $g$ are G\^ateaux differentiable at $p$ and $q$, respectively. Then the following are equivalent:
	\begin{enumerate}[{\rm (i)}]
		\item There exists a unique $x\in \HH$ such that $\pr_fx=p$ and $\pr_gx=q$.
		\item $p-q+\nabla f(p)-\nabla g(q)=0$.
	\end{enumerate}
\end{cor}
\begin{proof}
	Assume $f$ and $g$ are G\^ateaux differentiable at $p$ and $q$, respectively.
	
	(i)$\Leftrightarrow$(ii): Using Proposition \ref{Gatuni}, there exists an $x\in\mathcal{H}$ such that $x\in(p+\nabla f(p))\cap (q+\nabla g(q))$ if and only if $p+\nabla f(p)=x=q+\nabla g(q).$
	For the proof of uniqueness, let $y\in\mathcal{H}$ with the property that $\pr_fy=p \ \ \text{and} \ \ \pr_gy=q.$
	It follows that $y=p+\nabla f(p)=x.$
\end{proof}

A generalization of the classical orthogonal decomposition is  \emph{Moreau's decomposition} \cite[Remark 14.4]{ba}, which has also been extended to Banach spaces \cite{corey}. Specifically, for any $f\in\Gamma_0(\mathcal{H})$, and for every $x\in\HH$, $$x=\pr_fx+\pr_{f^\ast}x,$$
where $f^\ast$ denotes the Fenchel conjugate of $f$ (see \cite[Definition 13.1]{ba} for more details). A useful instance of this transform is the identity, $\iota^\ast_C=\sigma_C$, which we will employ frequently \cite[Example 13.3 (i)]{ba}. Furthermore, by \cite[Example 13.43 (i)]{ba}, for any nonempty, closed, and convex set $C$, its support function satisfies $\sigma^\ast_C=\iota_C$. 
\begin{rmk}
	\label{fenchelsubd}
	Let $f\in\Gamma_0(\mathcal{H})$. For any $x\in\HH$, $\pr_{f^\ast}x\in\partial f\left(\pr_fx\right)$.
\end{rmk}
\begin{proof}
	Let us denote $s:=\pr_{f^\ast}x$. By Moreau's decomposition, $s=x-\pr_fx$, that is, $x
	-s=\pr_fx$. With the relationship described in \eqref{proxsubd},
	\begin{center}
		$s=x-(x-s)\in\partial f(x-s)=\partial f\left(\pr_fx\right).$
	\end{center}    
\end{proof}

In the case that our Main Problem admits a solution to a system involving proximity operator of functions and their respective Fenchel conjugates, the Proposition \ref{Fenchuni} provides another sufficient condition that guarantees the uniqueness of the solution to the system \eqref{main_prob}.

\begin{prop}
	\label{Fenchuni}
	Let $(p_i)_{i\in\II}\in\bigtimes_{i\in \mathbb{I}}\dom\,\partial f_i$. If $f_{j}=f^\ast_k$, for some distinct indices $j,k\in\II$, then the system \eqref{main_prob}
	admits at most one solution in $\mathcal{H}$. In case there is a solution $x\in \mathcal{H}$ to \eqref{main_prob}, then $x=p_j+p_k$.
\end{prop}
\begin{proof}
	Let $f_j=f^\ast_k$, for some $j,k\in\II$ with $j\neq k$. Assume \eqref{main_prob} admits a solution $x\in\HH$. Since $f_j,f_k\in\Gamma_0(\HH)$, $f^\ast_j=f_k$. By Moreau's decomposition, $$x=\pr_{f_j}x+\pr_{f^\ast_j}x=\pr_{f_j}x+\pr_{f_k}x=p_j+p_k.$$
	Now, if \eqref{main_prob} admits another solution $y\in\HH$ then by Moreau's decomposition, $$y=\pr_{f_j}y+\pr_{f^\ast_j}y=\pr_{f_j}y+\pr_{f_k}y=p_j+p_k=x.$$
	Thus, if a solution to \eqref{main_prob} exist then it must be unique. Combining with the Proposition \ref{exist}, the system described in \eqref{main_prob} admits at most one solution.    
\end{proof}

As a remark, we take note that $\iota_C$ is not G\^ateaux differentiable at every boundary points in $C$. In light with the Example \ref{main_example}, $\iota_C$ is not G\^ateaux differentiable at $(1,0)$. However, Proposition \ref{Fenchuni} shows that the system given by $\pr_{\iota_C}x=(1,0)\ \mathrm{and}\ \pr_{\sigma_C} x=(1,1)$ admits a unique solution. Combined with Corollary \ref{exampleuni}, this observation indicates that the uniqueness of the solution is not sufficient to guarantee the Gâteaux differentiability of each function involved in the system.

As suggested by Proposition \ref{exist}, determining a point $x\in \bigcap_{i\in \mathbb{I}} \left( p_i + \partial f_i(p_i) \right)$ may be challenging. This leads naturally to a convex feasibility problem, which forms the focus of Section \ref{secfeas}. Our first theorem will address this by establishing necessary and sufficient conditions for the existence of solutions to the system \eqref{main_prob}. To this end, we recall a generalized form of Fenchel-Rockafellar duality, where the underlying assumption involves the interior ($\mathrm{int}$) of a set. We refer the reader to \cite[Eqn. (6.6)]{ba} for the definition of this notion of interior.

\begin{lem}{\rm \cite[{Eq. (1.1)}]{li}} \label{fenchrockthm} For each $i\in \mathbb{I}$, let $g_i\in \Gamma_0(\mathcal{H})$. Suppose that 
	\begin{equation}\label{condition}
		(\exists k\in \mathbb{I}) \ \mathrm{dom}\,g_k  \  \cap  \ \mathrm{int}\,\left( \bigcap_{i\in \mathbb{I}\setminus\{k\}} \mathrm{dom}\,g_i\right) \neq \varnothing.
	\end{equation}
	Then 
	\begin{equation}
		\label{fenchrock}
		\inf_{u\in \mathcal{H}} \left( \sum_{i\in\mathbb{I}} g_i(u)\right)  = -\min_{\{x_i\}_{i\in \mathbb{I}} \subset \mathcal{H}} \left\{\sum_{i\in \mathbb{I}} g_i^* (x_i)\, \Bigg|  \, \sum_{i\in\mathbb{I}} x_i=0\right\}.
	\end{equation}
\end{lem}

We also need the concept of linear independence of sets. A family of subsets $\{A_i\}_{i\in \mathbb{I}}$ is \emph{linearly independent} if 
\begin{equation}
	\label{linind}
	\left(\forall (a_i)_{i\in\mathbb{I}} \in \bigtimes_{i\in \mathbb{I}} A_i\right) \ \sum_{i\in \mathbb{I}} a_i=0\ \text{implies}\ (\forall i\in \mathbb{I}) \ a_i=0.
\end{equation}
We are now to give a necessary and sufficient conditions for the existence of solutions to the system \eqref{main_prob}.

\begin{thm}
	\label{main_theorem}
	For each $i\in \mathbb{I}$, let 
	\begin{center}
		$f_i\in \Gamma_0(\mathcal{H})$, $p_i \in \mathrm{dom}\,\partial f_i$, $\alpha_i\in \mathbb{R}\setminus \{0\}$,
	\end{center}
	and set
	\begin{center}
		$A_i=\alpha_i\left( p_i + \partial f_i(p_i)\right)$, and $g_i = \sigma_{A_i}$.
	\end{center}
	Suppose that $\sum_{i\in \mathbb{I}} \alpha_i = 0$, and that the functions $\{g_i\}_{i\in \mathbb{I}}$ satisfy Condition \eqref{condition}. Consider the following statements.
	\begin{enumerate}
		\item[{\rm(i)}] There exists $x\in \mathcal{H}$ such that $(\forall i\in \mathbb{I}) \ \pr_{f_i}x =p_i$.
		\item[{\rm(ii)}] $\inf\limits_{u\in \mathcal{H}}\, \sum\limits_{i\in \mathbb{I}} \left( \langle u,\alpha_ip_i\rangle + \sigma_{\partial f_i(p_i)}(\alpha_iu) \right)= 0$.
	\end{enumerate}
	Then {\rm(i)}\,$\Rightarrow$\,{\rm (ii)}. Moreover, if we assume that there exists $x\in \mathcal{H}$ and a sequence $\{\varepsilon_i\}_{i\in \mathbb{I}}$ of nonzero real numbers such that
	\begin{center}
		$(\forall i\in \mathbb{I}) \ \varepsilon_i x \in A_i$ with $\sum_{i\in \mathbb{I}} \varepsilon_i =0$,
	\end{center}
	then {\rm(i)} and {\rm (ii)} hold provided that $(\forall i\in \mathbb{I}) \ \alpha_i = \varepsilon_i.$
\end{thm}
\begin{proof}
	For $u\in \mathcal{H}$ and for $i\in \mathbb{I}$,
	\begin{align*}
		\langle u,\alpha_ip_i\rangle + \sigma_{\partial f_i(p_i)}(\alpha_iu) =\langle u,\alpha_ip_i\rangle + \sigma_{\alpha_i\partial f_i(p_i)}(u)=\sigma_{\alpha_i(p_i+\partial f_i(p_i))}(u)=g_i(u).
	\end{align*}
	Making use of Lemma \ref{fenchrockthm} yields 
	\begin{align}
		\inf\limits_{u\in \mathcal{H}}\, \sum\limits_{i\in \mathbb{I}} \left( \langle u,\alpha_ip_i\rangle + \sigma_{\partial f_i(p_i)}(\alpha_iu) \right) 
		=\inf_{u\in \mathcal{H}}\left(\sum\limits_{i\in \mathbb{I}} g_i(u) \right)\nonumber = -\min_{ \{x_i\}_{i\in\mathbb{I}} \subset \mathcal{H}} \left\{ \sum_{i\in \mathbb{I}} g_i^*(x_i)\, \Bigg| \,\sum_{i\in \mathbb{I}} x_i=0\right\} \nonumber
	\end{align}
	Consequently, by the Fenchel conjugate,
	\begin{align}
		\inf\limits_{u\in \mathcal{H}}\, \sum\limits_{i\in \mathbb{I}} \left( \langle u,\alpha_ip_i\rangle + \sigma_{\partial f_i(p_i)}(\alpha_iu) \right)
		= -\min_{ \{x_i\}_{i\in\mathbb{I}} \subset \mathcal{H}} \left\{ \sum_{i\in \mathbb{I}} \iota_{A_i}(x_i)\, \Bigg| \, \sum_{i\in \mathbb{I}} x_i=0\right\} =-\iota_{\sum_{i\in\mathbb{I}} A_i} (0).\label{indsum}
	\end{align}
	
	Now, suppose that (i) is true. By Proposition \ref{exist}, we can find an element $x$ in $ \bigcap_{i\in \mathbb{I}} \alpha_i^{-1}A_i$. Equivalently, for each $i\in \mathbb{I}$, $\alpha_ix\in A_i$. Moreover, $\sum_{i\in \mathbb{I}} \alpha_i x =0.$ Therefore (ii) follows.
	
	To prove the next assertion, suppose that there exists $x\in \mathcal{H}$ such that \begin{equation}(\forall i\in \mathbb{I}) \ \alpha_i x \in A_i.\label{indepen}\end{equation}  In view of \eqref{indsum} and since $\sum_{i\in \mathbb{I}} \alpha_i x = 0$, we get (ii). Finally, from \eqref{indepen}, we conclude that 
	$$x\in \bigcap_{i\in \mathbb{I}} (p_i + \partial f_i (p_i)).$$
	Therefore, (i) follows from Proposition \ref{exist}.    
\end{proof}

The implication (i)$\Rightarrow$(ii) of Theorem \ref{main_theorem} is a generalization of the \cite[Collorary 2.3]{co}. As a consequence of Theorem \ref{main_theorem}, we characterize those functions, $f$ and $g$, that satisfy the consistency of the system described in \eqref{main_prob}.
\begin{cor}
	\label{lem1}
	Let $f,g\in\Gamma_0(\mathcal{H})$, and let $(p,q)\in\emph{dom}\,\partial f\times\emph{dom}\,\partial g$. Set
	\begin{center}
		$G= \emph{dom}(\sigma_{-\partial g(q)})$ and $F=\emph{dom}(\sigma_{\partial f(p)})$.
	\end{center}
	Suppose that $G\cap \mathrm{int}(F)\neq\varnothing$. Then the following are equivalent:
	\begin{enumerate}
		\item[{\rm(i)}] There exists $x\in \mathcal{H}$ such that $\emph{Prox}_fx=p$ and $\emph{Prox}_gx=q$.
		\item[{\rm(ii)}] $\displaystyle\inf_{u\in\mathcal{H}} \left(\langle u,p-q\rangle+\sigma_{\partial f(p)}(u)+\sigma_{\partial g(q)}(-u)\right)=0$.
	\end{enumerate}
\end{cor}
\begin{proof}
	(i)$\Leftrightarrow$(ii): Let $A_1=p+\partial f(p)$ and $A_2=-\left(q+\partial g(q)\right)$. We conclude the proof by invoking Theorem \ref{main_theorem}.    
\end{proof}

In view of Corollary \ref{lem1} and \cite[Proposition 6.19]{ba}, the assumption $G\cap \mathrm{int}(F)\neq\varnothing$ can be replaced by any of the following: $G-F$ is a closed linear subspace or $0\in \mathrm{int}\,(G-F)$ or $0\in\mathrm{sri}(G-F)$ or $F\cap \mathrm{int}(G)\neq\varnothing$, where $G-F:=\{g-f\ \vert\ g\in G, f\in F\}$ and "$\mathrm{sri}$" stands for {\it strong relative interior} (see \cite[Definition 6.9]{ba}).

Given two convex sets, $C$ and $D$, let us investigate the solution of our system for a fixed $(p,q)\in C\times D$ using Corollary \ref{lem1}. The next two lemmas give us insight how tangent cones characterize the existence of solution to our system of projections onto convex sets.
\begin{lem}
	\label{lemsup} Let $A$ and $B$ be two nonempty sets. Then $\sigma_{A+B}= \sigma_A+\sigma_B.$
\end{lem}
\begin{proof}
	For each $u\in\mathcal{H}$,
	\begin{center}
		$\sigma_{A+B}(u)=\displaystyle\sup_{y=y_A+y_B\in A+B}\langle y,u\rangle\leq\displaystyle\sup_{y_A\in A}\,\langle y_A,u\rangle+\sup_{y_B\in B}\,\langle y_B,u\rangle= \sigma_A(u)+\sigma_B(u)$.
	\end{center}
	Let us show the other inequality. Fix $b\in B$. For each $u\in\mathcal{H}$,
	\begin{center}
		$\sigma_{A+b}(u)=\displaystyle\sup_{y_A\in A}\,\langle y_A+b,u\rangle=\displaystyle\sup_{y_A\in A}\,\langle y_A,u\rangle+\langle b,u\rangle= \sigma_A(u)+\langle b,u\rangle$.
	\end{center}
	Consequently, $\sup_{b\in B} \sigma_{A+b}(u)=\sup_{b\in B}(\sigma_A(u)+\langle b,u\rangle)=\sigma_A(u)+\sigma_B(u).$ But $A+b\subseteq A+B$ implies $\sigma_{A+b}\leq\sigma_{A+B}$. We then obtain the inequalities  $\sup_{b\in B}\sigma_{A+b}\leq\sigma_{A+B}$, and
	\begin{center}
		$\sigma_{A+B}(u)\geq\sup_{b\in B}\sigma_{A+b}(u)=\sigma_A(u)+\sigma_B(u)$, for each $u\in\mathcal{H}$.
	\end{center}
\end{proof}

In the local analysis of a convex set $C$, the normal cone operator $\N_C$ and the tangent cone operator $\T_C$ play fundamental roles; see \cite[Definition 6.38]{ba}. In particular, \cite[Proposition 1.2]{robval} shows that the translation of the normal cone at each proximal point is instrumental in characterizing the existence of solutions to the system \eqref{main_prob} when the underlying sets are convex. Motivated by this observation, we develop an alternative characterization of the existence of solutions to a system involving projections onto convex sets by exploiting properties of tangent cones. To this end, we introduce the notion of the polar cone, denoted by $C^\ominus$, which generalizes the orthogonal complement $C^\perp$; see \cite[Definition 6.22]{ba}. The following lemma will be used in the proof of Proposition \ref{convex}.
\begin{lem}
	\label{nortan}
	Let $C$ and $D$ be two nonempty closed convex sets. For any $(p,q)\in C\times D$, $$(\N_Cp-\N_Dq)^\ominus=\T_Cp\cap(-\T_Dq).$$
\end{lem}
\begin{proof}
	As mentioned in \cite[Proposition 6.44 (i)]{ba}, $\N_Cp=\T^{\ominus}_Cp$ and $\N_Dq=\T^{\ominus}_Dq$. Using the fact that normal cone operator is the polar cone of the tangent cone operator,
	$$\left(\N_Cp-\N_Dq\right)^\ominus=\left(\T_C^\ominus p-\T_D^\ominus q\right)^\ominus.$$ Invoking \cite[Proposition 6.27, Corollary 6.34]{ba}, we must obtain
	$$(\N_Cp-\N_Dq)^\ominus=\left(\T_Cp\cap(-\T_Dq)\right)^{\ominus\ominus}=\T_Cp\cap(-\T_Dq).$$    
\end{proof}

\begin{prop}
	\label{convex}
	Let $C$ and $D$ be nonempty closed convex subsets of $\mathcal{H}$ and let $(p,q)\in C\times D$. Suppose that $\T_Cp+\T_Dq$ is a closed linear subspace. Then the following are equivalent:
	\begin{enumerate}
		\item[{\rm(i)}] There exists $x\in \mathcal{H}$ such that $\p_Cx=p$ and $\p_Dx=q$.
		\item[{\rm(ii)}] $\T_Cp\cap (-\T_Dq)\subseteq \{q-p\}^\ominus$.
	\end{enumerate}
\end{prop}
\begin{proof}
	Set $f:=\iota_C$ and $g:=\iota_D$.
	Then $\mathrm{dom}\,\partial f=C$, $\mathrm{dom}\,\partial g=D$, $\partial f(p)=\N_Cp$, and $\partial g(q)=\N_Dq$. Meanwhile, since $p\in C$, $\N_Cp=(C-p)^\ominus$. By the \cite[Proposition 6.33]{ba},
	$$\dom\left(\sigma_{\N_Cp}\right)=\dom\left(\sigma_{(C-p)^\ominus}\right)=\dom\left(\iota_{(C-p)^{\ominus\ominus}}\right)=\dom\left(\iota_{\T_Cp}\right)=\T_Cp.$$
	Similarly, we obtain $\dom\left(\sigma_{-\N_Dq}\right)=-\T_Dq$. In view of Lemma \ref{lem1}, since $\dom\left(\sigma_{-\N_Dq}\right)-\dom\left(\sigma_{\N_Cp}\right)=-\left(\T_Dq+\T_Cp\right)$ is a closed linear subspace, there exists $x\in \mathcal{H}$ such that
	\begin{align*}
		\p_Cx =p \ \mathrm{and} \ \p_Dx=q\ \text{if and only if}\ \inf_{u\in\mathcal{H}}\left(\langle u,p-q\rangle+\sigma_{\N_Cp}(u)+\sigma_{\N_Dq}(-u)\right)=0.
	\end{align*}
	By using Lemma \ref{lemsup} and 
	\begin{align*}
		0=\inf_{u\in\mathcal{H}}\left(\langle u,p-q\rangle+\sigma_{\N_Cp}(u)+\sigma_{\N_Dq}(-u)\right)=\inf_{u\in\mathcal{H}}\left(\langle u,p-q\rangle+\sigma_{\N_Cp-\N_Dq}(u)\right).
	\end{align*}
	This follows that $0=\sup_{u\in\mathcal{H}}\left(\langle u,q-p\rangle-\sigma_{\N_Cp-\N_Dq}(u)\right).$ But by definition of a Fenchel conjugate, $\sigma^\ast_{\N_Cp-\N_Dq}(q-p)=0.$
	Using Lemma \ref{nortan},
	\begin{align*}
		\p_Cx =p \ \mathrm{and} \ \p_Dx=q\
		\text{if and only if}\ \iota_{\left[\N_Cp-\N_Dq\right]^{\ominus\ominus}}(q-p)=0.
	\end{align*}
	But, $\iota_{\left[\N_Cp-\N_Dq\right]^{\ominus\ominus}}(q-p)=0$ if and only if $\{q-p\}\subseteq\left[\T_Cp\cap(-\T_Dq)\right]^\ominus$.
	Thus, 
	\begin{align*}
		\p_Cx =p \ \mathrm{and} \ \p_Dx=q\ \text{if and only if}\ \T_Cp\cap(-\T_Dq)\subseteq\{q-p\}^\ominus.
	\end{align*}    
\end{proof}

As a direct consequence of this proposition, consider $p$ to be a point on the boundary of $C$ that is not a support point of $C$. This means $\N_Cp=\{0\}$ and so, $\T_Cp=\mathcal{H}$. By the Proposition \ref{convex}, the system has a solution if and only if $-\T_Dq\subseteq\{q-p\}^\ominus$. That is, $\{p-q\}\subseteq \T_D^\ominus q=\N_Dq.$ Since $\{p-q\}\subseteq\N_Dq$ means $p\in q+\N_Dq$, the system has a solution if and only if $q=\p_Dp$. This observation was discussed by Combettes and Reyes in \cite[Remark 1.2]{co}.\bigskip

In subsequent discussions, we denote $\widehat{X}$ as the closed linear subspace parallel to the closed affine subspace $X$. Moreover, for sets $A$ and $B$, we define $A+B:=\{a+b\ \vert\ a\in A, b\in B\}$. The following results are direct consequences of Proposition \ref{convex} in the particular settings of closed convex cones and closed affine subspaces.
\begin{exm}\rm
	\label{exampleprojector}
	In view of Proposition \ref{convex}, let us carefully investigate the system involving projection operators onto some special type of convex sets.
	\begin{enumerate}
		\item Let $K$ and $M$ be nonempty closed convex cones of $\HH$. For any $(p,q)\in K\times M$,
		\begin{center}
			$\T_Kp=\overline{K+\RR p}$ and $\T_Mq=\overline{M+\RR q}$, where $\RR p=\{rp\ \vert\ r\in\RR\}$ and $\RR q=\{rq\ \vert\ r\in\RR\}$.
		\end{center}
		If $\overline{K+\mathbb{R}p}+\overline{M+\mathbb{R}q}$ is a closed linear subspace of $\HH$ then the following are equivalent:
		\begin{enumerate}
			\item[{\rm(i)}] There exists $x\in \mathcal{H}$ such that $\p_Kx=p$ and $\p_Mx=q$.
			\item[{\rm(ii)}] $\overline{K+\mathbb{R}p}\cap-(\overline{M+\mathbb{R}q})\subseteq \{q-p\}^\ominus$.
		\end{enumerate}
		\item 	Let $X$ and $Y$ be nonempty closed affine subspaces of $\mathcal{H}$ such that $\widehat{X}+\widehat{Y}$ is a closed linear subspace of $\HH$. Then $\T_Xp=\widehat{X}$ and $\T_Yq=\widehat{Y}$. By applying Proposition \ref{convex}, there exists $z\in \mathcal{H}$ such that
		\begin{center}
			$\p_Xz=p$ and $\p_Yz=q$ if and only if $\widehat{X}\cap\widehat{Y}\subseteq\{q-p\}^\ominus$.
		\end{center}
		For $\widehat{X}\cap\widehat{Y}\subseteq\{q-p\}^\ominus$, $q-p\in\left(\widehat{X}\cap\widehat{Y}\right)^\ominus=\left(\widehat{X}\cap\widehat{Y}\right)^\perp$. Now, if $\langle q-p,0\rangle=0$, for any $x\in\widehat{X}\cap\widehat{Y}$, then $x\in\{q-p\}^\ominus$. That is, $\widehat{X}\cap\widehat{Y}\subseteq\{q-p\}^\ominus$. Then the following are equivalent:
		\begin{enumerate}
			\item[{\rm(i)}] There exists $z\in \mathcal{H}$ such that $\p_Xz=p$ and $\p_Yz=q$.
			\item[{\rm(ii)}] $q-p\in\left(\widehat{X}\cap\widehat{Y}\right)^\perp$.
		\end{enumerate}
	\end{enumerate}  
\end{exm}
The following example illustrates the application of the preceding results and shows that the solution set of the system \eqref{main_prob} is a convex subset of $\HH$.
\begin{exm}\rm
	Consider $\mathcal{P}_2[x]$ be the vector space of all polynomials of degree at most 2 defined on the interval $[-1,1]$ whose coefficients are real numbers endowed with the inner product
	\begin{center}
		$\langle f,g\rangle=\displaystyle\int^1_{-1}f(t)g(t)\ dt$, for $f,g\in\mathcal{P}_2[x].$
	\end{center}
	Let $X=\left\{bx+cx^2\in\mathcal{P}_2[x]|\ b,c\in\mathbb{R}\right\}$ and $Y=\left\{1+dx\in\mathcal{P}_2[x]|\ d\in\mathbb{R}\right\}$.
	It can be easily checked that $X$ and $Y$ are affine subspaces of $\mathcal{P}_2[x]$ with
	\begin{center}
		$\widehat{X}=X$ and $\widehat{Y}=\{dx\in\mathcal{P}_2[x]|\ d\in\mathbb{R}\}$.
	\end{center}
	Then $\widehat{X}+\widehat{Y}=X$ which is a closed linear subspace. Now, let us show that there exists a polynomial $h\in\mathcal{P}_2[x]$ such that
	\begin{center}
		$\p_Xh=x+x^2$ and $\p_Yh=1+x$.
	\end{center}
	We first observed that $\widehat{X}\cap\widehat{Y}=\widehat{Y}$. In view of part 2 of Example \ref{exampleprojector}, the existence of a solution is guaranteed if
	$$\widehat{Y}\subseteq\{(1+x)-(x+x^2)\}^\ominus=\{1-x^2\}^\ominus.$$
	Indeed,
	\begin{align*}
		\{1-x^2\}^\ominus=&\left\{u_1+u_2x+u_3x^2\in\mathcal{P}_2[x]\bigg\vert\ \displaystyle\int^1_{-1} (1-t^2)(u_1+u_2t+u_3t^2)\ dt\leq 0\right\}\\
		=&\left\{u_1+u_2x+u_3x^2\in\mathcal{P}_2[x]\big\vert\  5u_1+u_3\leq 0\right\}.
	\end{align*}
	It follows that $\widehat{Y}\subseteq\{1-x^2\}^\ominus.$ Henceforth, there is a solution to the system of equations. In fact, the set of all solution to the system of equations is given by the convex set
	$$\left\{\alpha_1+x+\alpha_2x^2\bigg\vert\ \dfrac{\alpha_1}{3}+\dfrac{\alpha_2}{5}=\dfrac{1}{5}\right\}.$$
\end{exm}\bigskip

So far, we have investigated the existence of an analytic solution to the system \eqref{main_prob}. However, in most cases, it is not easy to obtain an exact solution. In the worst case, the system might have no solution. For the purpose of numerical methods, we end this section by providing necessary and sufficient conditions for the existence of an approximate solution to the system \eqref{main_prob}. A sequence $\{x_n\}$ is said to be an \emph{approximate solutions} of the system described in \ref{main_prob} provided, for each $\varepsilon>0$, there exists $N\in\mathbb{N}$ such that
\begin{center}
	for any $n\geq N$, $\displaystyle\sum_{i\in\mathbb{I}}\lVert \pr_{f_i}x_n-p_i\rVert<\varepsilon.$
\end{center}\bigskip

The next proposition provides a characterization for the existence of solutions to the system \eqref{main_prob} in terms of a sequence of approximate solutions. In fact, one of the solutions of the system is the limit of the sequence of approximate solutions. 
\begin{prop}
	\label{approxconverge}
	For each $i\in \mathbb{I}$, let $f_i\in \Gamma_0(\mathcal{H})$ and $p_i\in \mathrm{dom}\,\partial f_i$. There exists a convergent sequence $\{x_n\}_{n\in\mathbb{N}}\subseteq\mathcal{H}$ which is an approximate solutions to the system
	\begin{center}
		$(\forall i\in\mathbb{I})\ \pr_{f_i}x=p_i$
	\end{center}
	if and only if the system has a solution.
\end{prop}
\begin{proof}
	Let $\varepsilon>0$. Since $x_n\rightarrow z$, there exists $N_1\in\mathbb{N}$ such that for any $n\geq N_1$, $\lVert z-x_n\rVert<\dfrac{\varepsilon}{2}.$
	By the nonexpansiveness of proximity operators, we have
	\begin{center}
		$(\forall i\in\mathbb{I})\ \lVert \pr_{f_i}z-\pr_{f_i}x_n\rVert\leq \lVert z-x_n\rVert<\dfrac{\varepsilon}{2}$, for any $n\geq N_1$.
	\end{center}
	Now, since $\{x_n\}$ is an approximate solutions, there exists an $N_2\in\mathbb{N}$ such that for any $n\geq N_2$,
	\begin{center}
		$(\forall i\in\mathbb{I})\  \lVert \pr_{f_i}x_n-p_i\rVert<\dfrac{\varepsilon}{2}$.
	\end{center}
	Take $N=\max\{N_1,N_2\}$ so that for all $n\geq N$,
	\begin{align*}
		(\forall i\in\mathbb{I})\  \lVert \pr_{f_i}z-p_i\rVert\leq \lVert \pr_{f_i}z-\pr_{f_i}x_n\rVert+\lVert\pr_{f_i}x_n-p_i\rVert<\dfrac{\varepsilon}{2}+\dfrac{\varepsilon}{2}=\varepsilon.
	\end{align*}
	Since $\varepsilon$ is arbitrary, $(\forall i\in\mathbb{I})\ \pr_{f_i}z=p_i.$
	Conversely, if the system has a solution, say $x$. We take the sequence $\{x_n\}_{n\in\mathbb{N}}=\{x\}.$ Consequently, the sequence $\{x_n\}$ is an approximate solutions with $x_n\rightarrow x$.    
\end{proof}

Suppose that system \eqref{main_prob} has no solution; however, we can find approximate solutions to the system. Then, Proposition \ref{approxconverge} tells us that the sequence of approximate solutions diverges. In other words, our Main Problem described by the system \eqref{main_prob} is an ill-posed problem. To be more precise, the next result shows that the growth norm of this sequence of approximate solutions will eventually blow up, which generalizes \cite[Proposition 2.1]{re}.

\begin{prop}
	\label{solblows}
	For each $i\in \mathbb{I}$, let $f_i\in \Gamma_0(\mathcal{H})$ and $p_i\in \mathrm{dom}\,\partial f_i$. Suppose the system 	\begin{center}
		$(\forall i\in\mathbb{I})\ \pr_{f_i}x=p_i$ has no solution.
	\end{center}
	If there exists an approximate solutions $\{x_n\}_{n\in\NN}$ to the system then $\Vert x_n\Vert\rightarrow +\infty$ as $n\rightarrow+\infty$.
\end{prop}
\begin{proof}
	For each $f_i\in\Gamma_0(\mathcal{H})$, $(\mathrm{Id}-\pr_{f_i})$ is a nonexpansive operator \cite[Proposition 12.28]{ba}. Then, the proof concludes by invoking \cite[Lemma 4.6]{robval} with $D=\mathcal{H}$.    
\end{proof}

\begin{prop}
	\label{prop_main}
	Let $f,g\in\Gamma_0(\mathcal{H})$. Let $(p,q)\in\emph{dom}\,\partial f\times\emph{dom}\,\partial g$. If 
	$$\displaystyle\inf_{u\in\HH}\ \langle u,p-q\rangle+\sigma_{\partial f(p)}(u)+\sigma_{\partial g(q)}(-u)=0$$
	then there exists an approximate solution $\{x_n\}_{n\in\NN}$ to the system
	\begin{center}
		$\emph{Prox}_fx=p$ and $\emph{Prox}_gx=q$.
	\end{center}
\end{prop}
\begin{proof}
	Given the assumption
	$$\displaystyle\inf_{u\in\HH}\ \langle u,p-q\rangle+\sigma_{\partial f(p)}(u)+\sigma_{\partial g(q)}(-u)=0,$$
	it must be the case that $\displaystyle\sup_{u\in\HH}\left(\langle u, q-p\rangle-\sigma_{\partial f(p)-\partial g(q)}(u)\right)=0$. Furthermore, we get $\iota_{\overline{\partial f(p)-\partial g(q)}}(q-p)=0.$
	This means $q-p\in \overline{\partial f(p)-\partial g(q)}$, that is, $0\in \overline{(p+\partial f(p))-(q+\partial g(q))}$. For each $\varepsilon>0$, we can find $x\in p+\partial f(p)$ and $y\in q+\partial g(q)$ such that $\Vert x-y\Vert<\varepsilon.$ Invoking \cite[Proposition 12.26]{ba} to $q=\pr_gy$ and $\pr_gx$,
	$$\langle \pr_gx-q,y-q\rangle+g(q)\leq g(\pr_gx)$$
	and $\langle q-\pr_gx,x-\pr_gx\rangle+g(\pr_gx)\leq g(q).$
	Using the two inequalities above,
	\begin{align*}
		\Vert \pr_gx-q\Vert^2&=\langle \pr_gx-q, \pr_gx-y\rangle+\langle \pr_gx-q, y-q\rangle\\
		&\leq\langle \pr_gx-q, \pr_gx-y\rangle+g(\pr_gx)-g(q)\\
		&=\langle q-\pr_gx, y-\pr_gx\rangle+g(\pr_gx)-g(q)\\
		&=\langle q-\pr_gx, y-x\rangle+\langle q-\pr_gx, x-\pr_gx\rangle+g(\pr_gx)-g(q)\\
		&\leq\langle q-\pr_gx, y-x\rangle.
	\end{align*}
	Finally, since $q=\pr_gy$ and knowing that proximity operator are firmly nonexpansive,
	$$\Vert \pr_gx-q\Vert^2\leq\langle \pr_gy-\pr_gx, y-x\rangle\leq\Vert y-x\Vert^2.$$
	Consider the sequence $\{x_n\}_{n\in\NN}$, where $x_n=x$ for all $n\in\NN$. We then conclude,
	\begin{center}
		$\Vert \pr_gx_n-q\Vert\leq\Vert y-x\Vert<\varepsilon$ and $\Vert \pr_fx_n-p\Vert=0<\varepsilon$.
	\end{center}
\end{proof}

Before presenting our second main result, we provide an analogous definition of locally constant mapping mentioned from \cite{vaja}, \cite[Definition 1.5]{fall}, and \cite[Example 1.0.3]{hart}.
A set-valued operator $T:\HH\rightarrow 2^{\HH}$ is said to be {\it locally constant} at $p\in\HH$, if there exists $\delta>0$ such that
\begin{center}
	$T(z)\subseteq T(p)$, for all $z\in B_{\delta}(p):=\{x\in\HH\ \vert\ \Vert x-p\Vert<\delta\}$.  
\end{center}

In light of Corollary \ref{lem1}, one may ask what occurs in the case where $G\cap\,\mathrm{int}(F)=\varnothing$. The following result shows that local constancy of the subdifferential provides a sufficient condition for statement (ii) in Corollary \ref{lem1} to characterize the existence of approximate solutions.
\begin{thm}
	\label{appro}
	Let $f,g\in\Gamma_0(\mathcal{H})$, and let $(p,q)\in\emph{dom}\,\partial f\times\emph{dom}\,\partial g$. If $\partial f$ and $\partial g$ are locally constant at $p$ and $q$, respectively, then the following statements are equivalent:
	\begin{enumerate}
		\item[{\rm(i)}] There exists an approximate solutions $\{x_n\}_{n\in\NN}$ to the system $\emph{Prox}_fx=p$ and $\emph{Prox}_gx=q$.
		\item[{\rm(ii)}] $\displaystyle\inf_{u\in\HH}\ \langle u,p-q\rangle+\sigma_{\partial f(p)}(u)+\sigma_{\partial g(q)}(-u)=0.$
	\end{enumerate}
\end{thm}
\begin{proof} For the proof of (i)$\Rightarrow$(ii), we use Proposition \ref{prop_main}. As for the converse statement, we first take note that by Lemma \ref{lemsup},
	$$\displaystyle\inf_{u\in\HH}\ \langle u,p-q\rangle+\sigma_{\partial f(p)}(u)+\sigma_{\partial g(q)}(-u)=-\displaystyle\sup_{u\in\HH}\ \langle u,q-p\rangle+\sigma_{\partial f(p)-\partial g(q)}(u).$$
	Let us proceed by contradiction. Suppose that (i) holds but
	\begin{eqnarray}
		\label{inf}
		\displaystyle\inf_{u\in\HH}\ \langle u,p-q\rangle+\sigma_{\partial f(p)}(u)+\sigma_{\partial g(q)}(-u)\neq 0.    
	\end{eqnarray}
	From \eqref{inf}, we see that $$\iota_{\overline{\partial f(p)-\partial g(q)}}(q-p)=\sigma^\ast_{\partial f(p)-\partial g(q)}(q-p)=\displaystyle\sup_{u\in\HH}\left(\langle u, q-p\rangle-\sigma_{\partial f(p)-\partial g(q)}(u)\right)\neq 0.$$
	With this, $q-p\notin \overline{\partial f(p)-\partial g(q)}$. As a consequence, there exists an $\varepsilon_0>0$ such that for all $v\in\partial f(p)$ and for all $w\in\partial g(q)$,
	\begin{equation}
		\label{closed}
		\Vert (v-w)-(q-p)\Vert\geq \varepsilon_0.
	\end{equation}
	But by the assumption (i) corresponding to $\varepsilon_0$, there exist a sequence $\{x_n\}\subseteq\HH$ and an $N_1\in\mathbb{N}$ for which
	\begin{equation}
		\label{approx}
		\Vert\pr_fx_n-p\Vert+\Vert\pr_gx_n-q\Vert<\varepsilon_0,\ \mathrm{whenever}\ n\geq N_1.
	\end{equation}
	Since $\partial f$ is locally constant at $p$, there exists $\delta>0$ such that $\partial f(z)\subseteq \partial f(p)$, for all $z\in B_{\delta}(p)$. Now, since $\pr_fx_n\rightarrow p$, corresponding to $\delta$, there exists $M_1\in\NN$ such that, for all $n\geq M_1$,
	\begin{center}
		$\Vert \pr_fx_n-p\Vert<\delta$, that is, $\pr_fx_n\in B_{\delta}(p)$.
	\end{center}
	Consequently, $\partial f\left(\pr_fx_n\right)\subseteq \partial f(p)$ for all $n\geq M_1$. But by Remark \ref{fenchelsubd}, $\pr_{f^\ast}x_n\in\partial f(p)$ for all $n\geq M_1$.
	Since $\partial g$ is locally constant at $q$ and $\pr_gx_n\rightarrow q$, applying similar approach yields $\pr_{g^\ast}x_n\in\partial g(q)$ for all $n\geq M_2$, for some $M_1\in\NN$.
	In view of \eqref{closed}, letting $N=\max\{N_1,M_1.M_2\}$ yields
	\begin{align*}
		\varepsilon_0\leq&\big\Vert \left(\pr_{f^\ast}x_N-\pr_{g^\ast}x_N\right)-(q-p)\big\Vert\\
		=&\big\Vert \left(x_N-\pr_fx_N\right)-\left(x_N-\pr_gx_N\right)-q+p\big\Vert\\
		\leq &\Vert p-\pr_fx_N\Vert+\Vert\pr_gx_N-q\Vert,
	\end{align*}
	a contradiction to the inequality \eqref{approx}.
	Hence, assertion (ii) must be true.
\end{proof}

We can use Theorem \ref{appro} to provide a characterization of approximate solutions of the system involving projection operators onto convex sets.
\begin{exm}\rm
	\label{example}
	The next two examples can be easily verified.
	\begin{enumerate}
		\item Let $C$ and $D$ be nonempty closed convex subsets of $\mathcal{H}$. Let $(p,q)\in C\times D$. If $\N_C\left(B_\delta(p)\right)\subseteq \N_C(p)$ and $\N_D\left(B_\delta(q)\right)\subseteq \N_D(q)$, for some $\delta>0$ then
		\begin{center}
			$(\forall \varepsilon>0) (\exists x\in\mathcal{H})\, \Vert\p_Cx-p\Vert+\Vert\p_Dx-q\Vert<\varepsilon\Leftrightarrow \langle u, p-q\rangle\geq 0, \forall u\in \T_C(p)\cap -\T_D(q)$.
		\end{center}
		The same equivalence statements will hold when $C$ and $D$ are replace as closed convex cones.
		\item Let $X$ and $Y$ be nonempty closed affine subspaces of $\mathcal{H}$. As shown in \cite[Example 6.43]{ba}, we deduce that $\N_X\left(B_\delta(p)\right)\subseteq \widehat{X}^\perp=\N_X(p)$ and $\N_Y\left(B_\delta(q)\right)\subseteq\widehat{Y}^\perp= \N_Y(q)$. Consequently, for any $(p,q)\in X\times Y$,
		\begin{center}
			$(\forall \varepsilon>0) (\exists x\in\mathcal{H})\, \Vert\p_Xx-p\Vert+\Vert\p_Yx-q\Vert<\varepsilon\Leftrightarrow \langle u, p-q\rangle= 0, \forall u\in \widehat{X}\cap\widehat{Y}$.
		\end{center}
	\end{enumerate}
\end{exm}

We end this section with a corollary showing that \cite[Proposition 2.5]{co} can be recovered using Theorem \ref{appro}, where the authors proved that linear independence is a necessary and sufficient condition for the existence of an approximate inverse best approximation property.
\begin{cor}
	Let $U$ and $V$ be nonempty closed linear subspaces of $\mathcal{H}$. Then the following are equivalent:
	\begin{enumerate}
		\item[\rm{(i)}] For each $(p,q)\in U\times V$ and for each $\varepsilon>0$, there exists $x\in\mathcal{H}$ such that
		$$\max\{\Vert\p_Ux-p\Vert, \Vert\p_Vx-q\Vert\}<\varepsilon.$$
		\item[\rm{(ii)}] $U\cap V=\{0\}$.
	\end{enumerate}
\end{cor}
\begin{proof}
	(ii)$\Rightarrow$ (i): As seen from part 2 in Example \ref{example}, if $U\cap V=\{0\}$, then (i) immediately follows since $\langle u,p-q\rangle=0$, for any $(p,q)\in U\times V$.\\
	(i)$\Rightarrow$ (ii): Since we know that $\{0\}\subseteq U\cap V$, we only have to show that $U\cap V\subseteq\{0\}$. Let $x\in U\cap V$. Invoking part 2 of Example \ref{example} with assumption (i), for each $(p,q)\in U\times V$, $\langle x,p-q\rangle=0$. This means $x\in (U-V)^\perp=(U+V)^\perp=U^\perp\cap V^\perp$. Since $U\cap U^\perp=\{0\}$ and $x\in U\cap U^\perp$, $x=0$. This concludes the proof of the corollary.    
\end{proof}

\section{The Inverse Proximal Property (IPP)}
\label{sec3}
In this section, we present a natural extension of the inverse best approximation property (IBAP), which was first coined by Combettes and Reyes \cite[Definition 1.1]{co}. Our objective here is to make analogous characterizations of the IBAP as presented in the paper \cite[Theorem 2.8]{co} in terms of proximity operator.

\begin{defn}\rm
	\label{ibap}
	Let $f,g\in\Gamma_0(\mathcal{H})$. Set $C:=\mathrm{dom\ }\partial f$ and $D:=\mathrm{dom\ }\partial g$. We say that $(C,D)$ satisfies the {\it inverse proximal property} (IPP) relative to $(f,g)$ if for any $(p,q)\in C\times D$, there exists $x\in\mathcal{H}$ such that
	$$\pr_fx=p \ \ \mathrm{and} \ \ \pr_gx=q.$$
\end{defn}

\begin{exm}\rm
	\label{exm:noibap}
	Take $\mathcal{H}=\mathbb{R}$. Consider the indicator function $f:=\iota_{[0,+\infty)}$ with the {\it rectified linear unit} activation function \cite[Example 2.6]{compes2}
	\begin{center}	$\rho:\mathbb{R}\rightarrow\mathbb{R}:x\mapsto\begin{cases}
			x, \hspace{0.5in}\mathrm{if\ }x>0\\
			0, \hspace{0.5in}\mathrm{if\ }x\leq0,
		\end{cases}$
	\end{center}
	and the {\it negative Boltzmann-Shannon entropy} \cite[Example 9.35]{ba}
	\begin{center}
		$g: \mathcal{H}\longrightarrow(-\infty,+\infty]: x \mapsto\begin{cases}
			x\ln(x)-x, \hspace{.2in} \text{if} \ \ x> 0\\
			0,\hspace{.74in} \text{if} \ \ x=0\\
			+\infty, \hspace{.59in} \text{if} \ \ x<0.
		\end{cases}$
	\end{center}
	Set $C:=\mathrm{dom\ }\partial f=[0,+\infty)$ and $D:=\mathrm{dom\ }\partial g=[0,+\infty)$. From \cite[Example 2.6]{compes}, $\pr_fx=\rho(x)$. It easy to show that
	$$\pr_gx=\begin{cases}
		\ln(x), \hspace{0.37in}\mathrm{if\ }x>0\\
		0, \hspace{0.58in}\mathrm{if\ }x\leq0.
	\end{cases}$$
\end{exm}
\noindent It can be observed that the system of equations, $\pr_fx=e$ and $\pr_gx=2$ has no solution. On the contrary, suppose there exists an $x\in\mathbb{R}$ such that $\pr_fx=e$ and $\pr_gx=2$. We then have, $\pr_fx=e$ if and only if $x=e$. Consequently, $\pr_g(e)=1$, which is a contradiction. This concludes that $(C,D)$ does not satisfy the IPP relative to $(\iota_{[0,+\infty)},g)$.

\begin{rmk}\rm
	From the Definition \ref{ibap}, if we choose $f:=\iota_C$ and $g:=\iota_D$, the pair $(C,D)$ {\it satisfies the IBAP} means $(C,D)$ satisfies the IPP relative to $(\iota_C,\iota_D)$.
\end{rmk}

One of our objectives in this section is to provide a characterization of the IPP. Analogously to the characterization provided by Combettes and Reyes \cite[Theorem 2.8]{co}, the next result allows to recover some equivalence statements from the characterization of IBAP for linear subspaces. We now present a characterization of the IPP relative to two functions, $f$ and $g$. Since the proximity operator is a natural extension of the projection operator, our characterization specializes to the classical characterization of the IBAP for linear subspaces.
\begin{prop}
	\label{ibapver2}
	Let  $f,g\in\Gamma_0(\mathcal{H})$. Set $C:=\mathrm{dom}\,\partial f$, $C^\ast:=\mathrm{dom}\,\partial f^\ast$, $D:=\mathrm{dom}\,\partial g$, and $D^\ast:=\mathrm{dom}\,\partial g^\ast$. Consider the following statements:
	\begin{enumerate}
		\item[{\rm(i)}] $\left(C,D\right)$ satisfies the IPP relative to $(f,g)$.
		\item[{\rm(ii)}] For any $p\in C$, there exists $x\in\mathcal{H}$ such that $\pr_fx=p$ and $\pr_gx=0$.
		\item[{\rm(iii)}] $\pr_f\left(D^\ast\right)=C$.
		\item[{\rm(iv)}] $\mathcal{H}=C^\ast+D^\ast-\pr_{f^\ast}\left(D^\ast\right)$.
	\end{enumerate}
	If $0\in D$, then $\rm{(i)}\Rightarrow\rm{(ii)}$, $\rm{(ii)}\Leftrightarrow\rm{(iii)}$, and $\rm{(iii)}\Rightarrow\rm{(iv)}$.
\end{prop}
\begin{proof}
	(i)$\Rightarrow$(ii): For any $p\in C$, $(p,0)\in C\times D$. Invoking assumption (i), the statement (ii) immediately holds.\\
	(ii)$\Rightarrow$(iii): Observe that $\pr_f\left(D^\ast\right)\subseteq \mathrm{ran}\left(\pr_f(\mathcal{H})\right)=C$. To show the other inclusion, let $p\in C$. By (ii), there exists $x\in\mathcal{H}$ such that $\pr_fx=p$ and $\pr_gx=0$. In view of \eqref{proxsubd}, $\pr_gx=0$ if and only if $x\in\partial g(0)$. Moreover, with the aid of \cite[Proposition 16.10]{ba}, $0\in\partial g^\ast(x)$. This means $x\in\mathrm{dom}\,\partial g^\ast$. Thus, there exists $x\in\mathcal{H}$ such that $p=\pr_fx\in \pr_f\left(D^\ast\right).$\\
	(iii)$\Rightarrow$(ii): Let $p\in C$. Then by (iii), there exists $x\in D^\ast$ such that $p=\pr_f x$. Since $x\in D^\ast$, $0\in\partial g^\ast(x)$. Consequently, $x\in\partial g(0)$, and yields, $\pr_gx=0$.\newline
	(iii)$\Rightarrow$(iv): By Moreau's decomposition and with the fact stated in \cite[Eq. (24.3)]{ba}, $$\mathcal{H}=\pr_f(\mathcal{H})+\pr_{f^\ast}(\mathcal{H})=C+C^\ast.$$
	But with the assumption (iii), $\mathcal{H}= \pr_f\left(D^\ast\right)+C^\ast.$
	
	Meanwhile, we already know that $C^\ast+D^\ast-\pr_{f^\ast}\left(D^\ast\right)\subseteq\mathcal{H}$. We end the proof by showing $\mathcal{H}\subseteq C^\ast+D^\ast-\pr_{f^\ast}\left(D^\ast\right)$. Now, let $x\in\mathcal{H}$. Then, we may write $x$ as follows
	\begin{center}
		$x=\pr_fy+z$, for some $y\in D^\ast$ and for some $z\in C^\ast$.
	\end{center}
	Applying Moreau's decomposition on $y$, we get
	$$x=y-\pr_{f^\ast}y+z\in D^\ast-\pr_{f^\ast}\left(D^\ast\right)+C^\ast.$$
\end{proof}

As a remark, the conditional statement (ii) $\Rightarrow$ (i) in Proposition \ref{ibapver2} is not necessarily true. To show this, let $\mathcal{H}=\RR$ and consider the negative rectified linear unit activation function define as
$$\varphi(x)=\begin{cases}
	0, \hspace{0.5in}\mathrm{if}\ x\geq 0\\
	x, \hspace{0.5in}\mathrm{if}\ x< 0.
\end{cases}
$$
It can be checked that $\pr_fx=\varphi(x)$, where $f:=\iota_{(-\infty,0]}$. Take the function $g$ as the negative Boltzmann-Shannon entropy. Then, $C=(-\infty,0]$ and $D=[0,+\infty)$. For each $p\in C$, the system of nonlinear equations $\pr_fx=p$ and $\pr_gx=0$ has a solution.
We claim that $(C,D)$ has no IPP relative to $f$ and $g$. For instance, consider the system of equations
\begin{center}
	$\pr_fx=-1$ and $\pr_gx=1$.
\end{center}
Note that $\pr_fx=-1$ implies $x=-1$. But $\pr_g(-1)=0$, a contradiction. Thus, the system of equations, $\pr_fx=-1$ and $\pr_gx=1$ has no solution. Whereas, a counterexample for the conditional statement (iv) $\Rightarrow$ (iii) in Proposition \ref{ibapver2} can be seen from \cite[Example 2.5]{robval}. The authors considered a system involving projection operators onto convex cones.

Having developed the necessary preliminary results, we now proceed to characterize the IPP. To establish such characterization, linearity of proximity operator is a sufficient condition. This generalizes the result presented in the papers \cite[Corollary 2.12]{co} and \cite[Proposition 2.4]{robval}.

\begin{thm}
	\label{ibapcharac}
	Let  $f,g\in\Gamma_0(\mathcal{H})$. $C:=\mathrm{dom}\,\partial f$, $C^\ast:=\mathrm{dom}\,\partial f^\ast$, $D:=\mathrm{dom}\,\partial g$, and $D^\ast:=\mathrm{dom}\,\partial g^\ast$. Suppose the following holds:
	\begin{enumerate}
		\item[\rm{(a)}] $C\cap C^\ast=\{0\}$ and $0\in D$.
		\item[\rm{(b)}] $\pr_f$ and $\pr_g$, are linear operators.
	\end{enumerate}
	Then the following statements are equivalent:
	\begin{enumerate}
		\item[\rm{(i)}] $\left(C,D\right)$ satisfies IPP relative to $(f,g)$.
		\item[\rm{(ii)}] For any $p\in C$, there exists $x\in\mathcal{H}$ such that $\pr_fx=p$ and $\pr_gx=0$.
		\item[\rm{(iii)}] $\pr_f\left(D^\ast\right)=C$.
		\item[\rm{(iv)}] $\mathcal{H}=C^\ast+D^\ast-\pr_{f^\ast}\left(D^\ast\right)$.
	\end{enumerate}
\end{thm}
\begin{proof}
	From the previous theorem, we are only left to show the two implications, (ii) $\Rightarrow$ (i) and (iv) $\Rightarrow$ (iii) of Proposition \ref{ibapver2}.\\
	(ii) $\Rightarrow$ (i): Let $(p,q)\in C\times D$. Then, there exist $x,y\in\mathcal{H}$ such that $p=\pr_fx$ and $q=\pr_gy$. By the linearity of proximity operators, $p-\pr_fy= \pr_f x-\pr_fy=\pr_f\left(x-y\right)\in C.$ Using assumption (ii) corresponding to $p-\pr_fy$, we can find a vector $z\in\mathcal{H}$ for which $\pr_fz= p-\pr_fy$ and $\pr_gz=0.$
	As we can see
	\begin{center}
		$\pr_f(z+y)= \pr_fz+\pr_fy=p$ and $\pr_g(z+y)= \pr_gz+\pr_gy=q$.
	\end{center}
	This shows that $z+y$ is a solution to the system $\pr_fx=p$ and $\pr_gx=q$. Since $p$ and $q$ are chosen arbitrarily, $(C,D)$ satisfies the IPP relative to $(f,g)$.\\
	(iv) $\Rightarrow$ (iii): It is known that $\pr_f\left(D^\ast\right)\subseteq C$. It remains to show that $C\subseteq\pr_f\left(D^\ast\right)$. Let $p\in C$. By assumption (iii),
	\begin{center}
		$p=q+y-\pr_{f^\ast}z$, for some $q\in C^\ast$ and for some $y,z\in D^\ast$.
	\end{center}
	Since $q\in C^\ast$, $q=\pr_{f^\ast}x$, for some $x\in\mathcal{H}$. Take note that by Moreau's decomposition and by linearity of $\pr_{f^\ast}$,
	$$p-\pr_fy=\pr_{f^\ast}x+\pr_{f^\ast}y-\pr_{f^\ast}z=\pr_{f^\ast}(x+y-z)\in C^\ast.$$
	But one can checked that $p-\pr_fy\in C$. This yields to, $p-\pr_fy\in C\cap C^\ast=\{0\}$ and showing that $p=\pr_fy\in\pr_f\left(D^\ast\right)$.
\end{proof}

A classical example of a linear proximity operator is the projection operator onto a closed linear subspace. As a result, Theorem \ref{ibapcharac} plays an important role in the recovery of the characterization of the IBAP for linear subspaces, as discussed in \cite{co}. We shall see this in the following corollary.

\begin{cor}\rm
	\label{IBAPlinear}
	Let $U$ and $V$ be nonempty closed linear subspaces of $\mathcal{H}$. Recall that $\p_U$ and $\p_V$ are linear operators if and only if $U$ and $V$ are both linear subspaces. In line with the Theorem \ref{ibapcharac}, the following statements are equivalent:
	\begin{enumerate}
		\item[(i)] $\left(U,V\right)$ satisfies the IBAP.
		\item[(ii)] For any $p\in U$, there exists $x\in\mathcal{H}$ such that $\p_Ux=p$ and $\p_Vx=0$.
		\item[(iii)] $\p_U\left(V^\perp\right)=U$.
		\item[(iv)] $\mathcal{H}=U^\perp+V^\perp-\p_{U^\perp}(V^\perp)$.
	\end{enumerate}
	But (iv) can be further deduced as follows, $\mathcal{H}=U^\perp+V^\perp-\p_{U^\perp}(V^\perp)=U^\perp+V^\perp$.
	Thus, we recover the first four equivalent statements in the characterizations of the IBAP for linear subspaces \cite[Corollary 2.12]{co}.
\end{cor}

We end the section by providing a linear proximity operator, which need not be a projection operator.
\begin{exm}\rm
	In \cite[Example 13.6]{ba}, $\mathfrak{q}^\ast=\mathfrak{q}$. From Moreau's decomposition, $\pr_{\mathfrak{q}}=\dfrac{1}{2}\mathrm{Id}$ which further means that $\pr_{\mathfrak{q}}$ is a linear proximity operator. Another interesting linear proximity operator is when we consider the quadratic function $f(x)=\dfrac{1}{2}x^\top Qx+b^\top x$, where $Q$ is a positive semi-definite matrix. That is, $\pr_f x= (I+Q)^{-1}(x-b)$.
\end{exm}

\section{Applications}
\label{sec4}
\subsection{Feasibility Problem}
\label{secfeas}
\ \indent Let $m\in\mathbb{N}$. Given a family of closed convex sets $(C_i)_{1\leq i\leq m}$ of $\mathcal{H}$, the convex feasibility problem \cite{convexfeas} consists of finding $x\in\displaystyle\bigcap^m_{i=1}C_i$. In applications, this problem arises in constrained optimization problem where we want to find a point which satisfies the constraints. This problem has wide applicability in physical science, economics, prediction theory, image reconstruction, and signal processing. In \cite{ba2}, the authors showed that whenever $A\cap B\neq\emptyset$, it is necessary that $A\cap B$ is precisely the projection of $\mathrm{Fix}\ T$ onto $B$, where $T=\dfrac{1}{2}\left(R_AR_B+\mathrm{Id}\right)$ is the average reflectors with respect to $A$ and $B$ \cite[Corollary 3.9]{ba2}. In \cite{feasibility}, the authors proposed a new algorithm for solving convex feasibility problems, in which randomly selected blocks of subgradient projectors are activated in parallel at each iteration and combined through an extrapolation-based averaging process. They demonstrated the effectiveness of their method through numerical experiments in signal and image recovery.

The next lemma will aid us in providing a characterization for the convex feasibility problem to admit a solution.

\begin{lem}
	\label{subdsup0}
	\emph{\cite[Example 16.34]{ba}} Let $C$ be a nonempty closed convex subset of $\mathcal{H}$. Then $\partial\sigma_C(0)=C$.
\end{lem}

\begin{prop}
	\label{convfeas}
	Let $C$ and $D$ be two nonempty closed convex sets of $\mathcal{H}$. Set
	\begin{center}
		$G:=\{u\in\mathcal{H}\vert\ \sigma_{-D}(u)<+\infty\}$ and $F:=\{u\in\mathcal{H}\vert\ \sigma_{C}(u)<+\infty\}$.
	\end{center}
	If $G-F$ is a closed linear subspace of $\HH$ then
	\begin{center}
		$C\cap D\neq \varnothing$ if and only if $\displaystyle\inf_{u\in\mathcal{H}}\sigma_{C-D}(u)=0$.
	\end{center}
\end{prop}
\begin{proof}
	With the use of Lemma \ref{subdsup0}, $C\cap D=\partial\sigma_C(0)\cap\partial\sigma_D(0)$. In addition to (\ref{proxsubd}), we have $x\in C\cap D$ if and only if $\pr_{\sigma_{C}}x=0$ and $\pr_{\sigma_{D}}x=0$.
	In view of Corollary \ref{lem1}, take $f=\sigma_C$ and $g=\sigma_D$. Since $C\neq\varnothing$ and $D\neq\varnothing$, $0\in\mathrm{dom}\,\partial f\cap \mathrm{dom}\,\partial g$. It follows from Corollary \ref{lem1}, if $G-F$ is a closed linear subspace, then
	$$\pr_{\sigma_{C}}x=0\ \mathrm{and}\ \pr_{\sigma_{D}}x=0 \Leftrightarrow \inf_{u\in\mathcal{H}}\left(\sigma_C(u)+\sigma_{-D}(u)\right)=0.$$
	Therefore, invoking Lemma \ref{lemsup} yields, $C\cap D\neq \varnothing$ if and only if $\displaystyle\inf_{u\in\mathcal{H}}\sigma_{C-D}(u)=0$.    
\end{proof}

\noindent The Proposition \ref{convfeas} guarantees the existence of $x\in C\cap D$. To find such $x$ numerically, we may consider the periodic projection method (also known as POCS algorithm) discussed in \cite[Corollary 5.26]{ba}.
Here is an algorithm for two convex sets:\bigskip

\begin{algorithm}[H]
	\KwData{Given convex sets $C$ and $D$}
	\KwResult{Construct a sequence $(x_n)_{n\in\mathbb{N}}$ such that $x_n\rightharpoonup x\in C\cap D$} ($x_n$ converges weakly to $x$) initialize $x_0$\;
	\For{$n=0,1,2,...$}{
		$x_{n+1}=\p_C\p_Dx_n$\;
	}
\end{algorithm}
\bigskip

\noindent Take note that whenever $C$ and $D$ are either nonempty closed convex cones or nonempty closed linear subspaces, $C\cap D\neq\varnothing$. But for two nonempty closed affine subspaces $X$ and $Y$, $X\cap Y\neq\varnothing$ may fail. In \cite[Proposition 2.10]{co}, $(\widehat{X}^\perp, \widehat{Y}^\perp)$ has IBAP is a sufficient condition for the affine feasibility problem to have a solution. As a corollary to our Proposition \ref{convfeas}, we present a necessary and sufficient condition when an affine feasibility problem admits a solution.

\begin{cor}
	Let $X$ and $Y$ be nonempty closed affine subspaces of $\mathcal{H}$ such that $X=a+\widehat{X}$ and $Y=b+\widehat{Y}$, for some $(a,b)\in X\times Y$. If $\widehat{X}^\perp+\widehat{Y}^\perp$ is closed in $\HH$ then 
	\begin{center}
		$X\cap Y\neq\varnothing$ if and only if $a-b\in\widehat{X}+\widehat{Y}$.
	\end{center}
\end{cor}
\begin{proof}
	Consider $G$ and $F$ defined in Proposition \ref{convfeas} with $C=X$ and $D=Y$. Then $G-F=\widehat{X}^\perp+\widehat{Y}^\perp$. Consequently, if $\widehat{X}^\perp+\widehat{Y}^\perp$ is closed then 
	\begin{center}
		$X\cap Y\neq\varnothing$ if and only if  $\displaystyle\inf_{u\in\mathcal{H}}\sigma_{X-Y}(u)=0.$
	\end{center}
	Meanwhile,
	\begin{align*}
		\inf_{u\in\mathcal{H}}\sigma_{X-Y}(u)=\inf_{u\in\mathcal{H}}\left(\langle a,u \rangle+\sigma_{\widehat{X}}(u)+\langle -b,u\rangle+\sigma_{\widehat{Y}}(u)\right)=\inf_{u\in\mathcal{H}}\left(\langle a-b,u \rangle+\imath_{\widehat{X}^\perp\cap\widehat{Y}^\perp}\right).
	\end{align*}
	Indeed, $\inf_{u\in\mathcal{H}}\sigma_{X-Y}(u)=\inf_{u\in\widehat{X}^\perp\cap\widehat{Y}^\perp}\langle a-b,u \rangle$. Since we know that $\widehat{X}^\perp\cap\widehat{Y}^\perp=\widehat{X}+\widehat{Y}$,
	\begin{center}
		$\displaystyle\inf_{u\in\mathcal{H}}\sigma_{X-Y}(u)=0$ if and only if $a-b\in\widehat{X}+\widehat{Y}.$
	\end{center}    
\end{proof}

As a remark to this corollary, if $(\widehat{X}^\perp,\widehat{Y}^\perp)$ has the IBAP then $X\cap Y\neq\varnothing$. This was already proved by Combettes and Reyes in \cite[Proposition 2.10]{co}.

\subsection{Signal Recovery}
Let us consider the best nonlinear approximation problem proposed by Combettes and Woodstuck. In view of \cite[Problem 1.1]{comwood}, we take the firmly nonexpansive operator $\pr_f$, $f\in\Gamma_0(\mathcal{H})$. To investigate their problem, we shall see in the next result that the set of all solutions in our system is a convex set.

For $f,g\in\Gamma_0(\mathcal{H})$ and $(p,q)\in\mathrm{dom}\,\partial f\times\mathrm{dom}\,\partial g$, define the set
$$S^{(f,g)}_{(p,q)}=\{x\in\mathcal{H}\vert\ \pr_fx=p\ \mathrm{and}\ \pr_g=q\}.$$

\begin{lem}
	For each $(p,q)\in\mathrm{dom}\,\partial f\times\mathrm{dom}\,\partial g$, $S^{(f,g)}_{(p,q)}$ is a closed convex set.
\end{lem}
\begin{proof}
	Let $(p,q)\in\mathrm{dom}\,\partial f\times\mathrm{dom}\,\partial g$. Invoking Proposition \ref{exist},
	\begin{align*}
		S^{(f,g)}_{(p,q)}=(p+\partial f(p))\cap (q+\partial g(q)).
	\end{align*}
	But by \cite[Proposition 16.4(iii)]{ba}, $\partial f(p)$ and $\partial g(q)$ are closed convex sets in $\HH$. This further means $(p+\partial f(p))$ and $(q+\partial g(q))$ are also closed convex sets. Consequently, $S^{(f,g)}_{(p,q)}$ is a closed convex set.   
\end{proof}

In the next result, we shall see that a solution in \cite[Problem 1.1]{comwood} is precisely the projection onto a certain closed convex set. The given problem below can be interpreted as finding a best nonlinear approximation $x$ near the original signal $x_0$ in a Hilbert space $\mathcal{H}$ from a prescribed proximal point $p$ and some \emph{a priori} constraints on $x$ available from a countable family of closed convex sets $(C_j)_{j\in J}$.
\begin{prob}\rm
	\label{specialwood}
	Let $x_0\in\mathcal{H}$ and $f\in\Gamma_0(\mathcal{H})$. Let $(C_j)_{j\in J}$ be a countable family of nonempty closed convex subsets of $\mathcal{H}$. Set $C=\cap_{j\in J} C_j$. Given $p\in\mathrm{dom}\,\partial f$, find $x\in\mathcal{H}$ such that
	\begin{align}
		\label{min2}	&\min_{x\in\mathcal{H}} \Vert x-x_0\Vert\ \\
		\label{cons2}	&\mathrm{subject\ to}\ x\in C\ \mathrm{and}\ \pr_fx=p.
	\end{align}
	By \cite[Example 3.2(iv)]{ba}, $C$ is a closed convex set in $\HH$. Then, the convex constraints described in (\ref{cons2}) can be viewed as a system of nonlinear equation as follows
	\begin{equation*}
		\pr_{\sigma_C}x=0\ \mathrm{and}\ \pr_fx=p.
	\end{equation*}
	In view of Proposition \ref{exist}, Corollary \ref{lem1}, and Proposition \ref{ibapver2}, the recovery problem admits a solution if and only if one of the following holds:
	\begin{enumerate}
		\item[\rm{(i)}] $S^{(\sigma_C,f)}_{(0,p)}\neq\emptyset$;
		\item[\rm{(ii)}] $\dom\,\sigma_{\partial f(p)}\cap \mathrm{int}\left(\T_Cp\right)\neq\varnothing$ and $\inf_{u\in\mathcal{H}}(\sigma_{p+\partial f(p)}(u)+\sigma_C(-u))=0$;
		\item[\rm{(iii)}] $\pr_f(C)=\mathrm{dom}\,\partial f$.
	\end{enumerate}
	Then, the minimization problem in $(\ref{min2})$ with the constrained in $(\ref{cons2})$ can be reduced to finding $x\in\mathcal{H}$ such that 
	$$\min_{x\in S^{(\sigma_C,f)}_{(0,p)}}\Vert x-x_0\Vert.$$
	Since $\min_{x\in S^{(\sigma_C,f)}_{(0,p)}}\Vert x-x_0\Vert=\p_{S^{(\sigma_C,f)}_{(0,p)}}x_0$, $\displaystyle\p_{S^{(\sigma_C,f)}_{(0,p)}}x_0$ is the best nonlinear approximation near the signal $x_0$. To find $x\in\mathcal{H}$ numerically, we may consider the following Dykstra's algorithm \cite[Theorem 30.7]{co} applied to $C_j$ ($\forall j\in J$) and $p+\partial f(p)$:\bigskip
	
	\begin{algorithm}[H]
		\KwData{Let $\mathrm{card}(J)$ denotes the cardinality of $J$ and set $\omega=\mathrm{card}(J)+1$. For $0\leq i\leq \omega-1$, set $q_{-i}=0$. Consider the convex sets $(C_i)_{i\in J}$ and $C_\omega:=p+\partial f(p)$}
		\KwResult{Construct a sequence $(x_n)_{n\in\mathbb{N}}$ such that $x_n\rightarrow \p_{S^{(\sigma_C,f)}_{(0,p)}}x_0$}
		initialize $x_0$\;
		\For{$n=1,2,...$}{
			take $i=1+(n-1)\%\omega$\;
			$x_n=\p_{C_i}(x_{n-1}+q_{n-\omega})$\;
			$q_n=x_{n-1}+q_{n-\omega}-x_\omega$\;
		}
	\end{algorithm}	
\end{prob}

\section*{Acknowledgement}
The author gratefully acknowledges {\bf Noli N. Reyes} (1963-2020) for introducing the problem that motivated this study and sincerely thanks {\bf Louie John D. Vallejo} for valuable discussions and insightful comments that contributed to the development of this work.
\end{document}